\title{The random graph intuition for the tournament game}
\author{
\quad{Dennis Clemens
\thanks{Department of Mathematics and Computer Science, Freie Universit\"{a}t Berlin, Germany. Email: d.clemens@fu-berlin.de.
Research supported by DFG, project SZ 261/1-1.}}
\quad{Heidi Gebauer
\thanks{Institute of
Theoretical Computer Science, ETH Zurich, CH-8092 Switzerland. Email:
gebauerh@inf.ethz.ch. }
}
\quad{Anita Liebenau
\thanks{Department of Mathematics and Computer Science, Freie Universit\"{a}t Berlin, Germany. Email: liebenau@math.fu-berlin.de.
The author is supported by a scholarship from the Berlin Mathematical School
(BMS).}}}

\documentclass[11pt,a4paper]{article}
\usepackage{amsmath,amssymb,latexsym,color,epsfig,a4, enumerate}
\usepackage{bbm, mathtools}
\usepackage{a4wide}

\date{}  

\newtheorem{theorem}{Theorem}[section]
\newtheorem{lemma}[theorem]{Lemma}
\newtheorem{proposition}[theorem]{Proposition}

\newtheorem{corollary}[theorem]{Corollary}

\newtheorem{problem}[theorem]{Problem}

\newtheorem{claim}[theorem]{Claim}

\def\cf{{\cal F}}
\def\cT{{\cal T}}
\def\cS{{\cal S}}

\def\cH{{\cal H}}

\newcommand{\nat}{\ensuremath{\mathbbm{N}}}

\newenvironment{proof}{\noindent{\bf Proof\,}}{\hfill$\Box$}

\definecolor{green}{RGB}{0,100,0}

\begin{document}
\maketitle

\begin{abstract}
In the {\em tournament game} two players, called Maker and Breaker, alternately take turns in claiming an unclaimed edge
of the complete graph $K_n$ and selecting one of the two possible orientations. Before the game starts, Breaker
fixes an arbitrary tournament $T_{k}$ on $k$ vertices. Maker
wins if, at the end of the game, her digraph contains a copy of $T_{k}$; otherwise Breaker wins.
In our main result, we show that Maker has a winning strategy for
$k = (2-o(1))\log_2 n$, improving the constant factor in previous results
of Beck and the second author.
This is asymptotically tight since it is known that for $k = (2-o(1))\log_2 n$
Breaker can prevent that the underlying graph of Maker's graph contains a $k$-clique.
Moreover the precise value of our lower bound differs from the upper
bound only by an {\em additive constant} of $12$.

We also discuss the question whether the random graph intuition, which suggests that the threshold for $k$ is asymptotically the same for
the game played by two ''clever'' players and the game played by two ''random''� players, is supported by the tournament game: It will turn out that, while a straightforward application of this intuition fails, a more subtle version of it is still valid.

Finally, we consider the \emph{orientation-game} version of the tournament game, where Maker wins the game if the final digraph -- containing also the edges directed by Breaker -- possesses a copy of $T_{k}$. We prove that in that game Breaker has a winning strategy for $k = (4+o(1))\log_2 n$.
\end{abstract}
\normalsize

\section{Introduction}

Let $X$ be a finite set and let $\cf \subseteq 2^X$ be a family of subsets. In the classical Maker--Breaker game $(X,\cf)$, two players, called Maker and Breaker,
alternately claim elements of $X$, with Maker going first. $X$ is usually called the
{\em board}, and $\cf$ is referred to as the {\em family of winning sets}. Maker wins the game if she claims all elements of some winning set; otherwise Breaker wins.
A well-studied class of Maker--Breaker games are \emph{graph games}, where the board is the edge set of a complete graph $K_{n}$, and Maker's goal is to create a graph which possesses some fixed (usually monotone) property $P$.
A widely investigated example of such
a game
is the \emph{$k$-clique game} (sometimes abbreviated by \emph{clique game}) where Maker wins if and only if, by the end of the game, her graph contains a clique of size at least $k$.
In \cite{es1973}, Erd\H{o}s and Selfridge considered the largest value $k_{cl}=k_{cl}(n)$ such that Maker has a winning strategy in the $k_{cl}$-clique game. By applying their well-known Erd\H{o}s-Selfridge criterion, they obtained that $k_{cl} \leq (2-o(1))\log n$ (throughout this paper all logarithms are in base 2, unless stated otherwise).
Later, Beck  \cite{BeckBook} introduced the method of \emph{self-improving potentials} and used his technique to investigate  $k_{cl}$. Finally, he determined the exact value of $k_{cl}$!
\begin{theorem}\label{BeckCliqueTheo} (\cite[Theorem 6.4]{BeckBook})
$k_{cl} = \left\lfloor 2 \log n - 2 \log \log n +2\log e -3 +o(1) \right\rfloor$
\end{theorem}

\paragraph{Random Graph Intuition}
There is an intriguing relation between $k_{cl}$ and the corresponding extremal value $k^{\ast}_{cl}$ for a game where Maker and Breaker are replaced with ''random players'' which select their edge in each round completely at random: In this game, \emph{RandomMaker} creates a random graph $G(n,m)$ with $\lceil m = \frac{1}{2}\binom{n}{2} \rceil$ edges. It is well-known that the size of the largest clique of $G(n,\frac{1}{2}\binom{n}{2})$ is $(2 - o(1))\log n$ asymptotically almost surely (abbreviated as a.a.s. in the sequel), so the threshold where the random $k$-clique game turns from a RandomMaker's win to a RandomBreaker's win is around $(2 - o(1))\log n$; just like in the deterministic game, as shown by Theorem \ref{BeckCliqueTheo} .

For quite a few other games it has been found that the outcome of the random game is essentially the same as the outcome of the deterministic game (see, e.g., \cite{Beck94, BL01, GS09, HKSS08, ss2005}).  This phenomenon, known as the {\em random graph intuition} or the {\em Erd\H{o}s paradigm}, was first pointed out by  Chv\'{a}tal and Erd\H{o}s \cite{CE}, and later investigated further in many papers of Beck \cite{Beck85, Beck93, Beck94, Beck96} and Bednarska and \L uczak \cite{BL01}.

For a small number of games it has been established that the random graph intuition fails: An interesting example is the {\em diameter game} where the graph property $P$ Maker aims to achieve is that every pair of vertices has distance at most two (in her graph). It is known that a.a.s. $G(n, \lceil \frac{1}{2} \binom{n}{2}\rceil)$ has property $P$, hence RandomMaker wins the random game a.a.s. On the other hand, Balogh, Martin and Pluh\'{a}r \cite{BMP09} proved that actually Braker has a winning strategy in the deterministic game, which yields that the random graph intuition fails in this case. It is an interesting open problem to determine suitable criteria which guarantee for a given game that the random graph intuition (or some weaker version of it) holds.

\paragraph{Tournament Game}
In this paper we study a variant of the $k$-clique game and investigate whether it supports the random graph intuition.
A \emph{tournament} is a directed graph where every pair of vertices is connected by a single directed edge. The \emph{$k$-tournament game} $\cT(k,n)$ is played on $K_{n}$. At the beginning of the game Breaker fixes an arbitrary tournament $T_{k}$ on $k$ vertices. In each round, Maker and Breaker then alternately claim one unclaimed edge (as in classical graph games), and -- additionally -- select one of the two possible orientations for their chosen edge. If, at the end of the game, Maker's graph contains a copy of the goal tournament $T_k$, she wins; otherwise, Breaker is the winner. Note that for the outcome of this particular game, the orientations of Breaker's edges are irrelevant.
In the light of general {\em orientation games}, which we shall introduce shortly, they become meaningful though.

Let $k_{t} = k_{t} (n)$ denote the largest $k$ such that Maker has a winning strategy in the game $\cT(k,n)$. To get an indication for the value of $k_{t}$, Beck analyzed the {\em random tournament game} in which RandomMaker and RandomBreaker each choose their edge and the corresponding orientation uniformly at random. He found that the threshold where the random game turns from a RandomMaker's win to a RandomBreaker's win is around $(1 - o(1)) \log n$. Motivated by the question whether the tournament game supports the random graph intuition, Beck \cite{BeckBook} asked to determine $k_{t}$.
Since a winning strategy for Breaker in the $k$-clique game allows him to prevent Maker from achieving any tournament on $k$ vertices, we have
\begin{equation} \label{eq:upperboundktour}
k_{t} \leq k_{cl} = \left(2-o(1)\right)\log n.
\end{equation}
The second equation follows from Theorem \ref{BeckCliqueTheo}. From the other side, Beck \cite[p. 457]{BeckBook} derived that \[k_{t} \geq \left(\frac{1}{2}-o(1)\right)\log n.\]  In fact, he proved the stronger statement that for $k = (\frac{1}{2} - o(1))\log n$, Maker has a strategy to occupy a graph containing a copy of \emph{every} tournament on $k$ vertices. The lower bound on $k_{t}$ was improved by the second author in \cite{g2012} to $k_{t} \geq (1 - o(1)) \log n$.
In our main result we show that the upper bound is tight. This means that $k_{t}$ is twice as large as the random graph intuition suggests.
\begin{theorem} \label{MainTournament}
$k_{t} \geq 2\, \log n - 2\, \log \log n - 12 = (2 - o(1))\log n$.
\end{theorem}
As a direct consequence of \eqref{eq:upperboundktour}, Theorem \ref{BeckCliqueTheo} and Theorem \ref{MainTournament}, the asymptotics of $k_{t}$ are determined:
\[k_{t} = 2 \log n - 2 \log \log n + \Theta(1) = (2 - o(1)) \log n. \]
Remarkably, the upper bound in the clique-game and the lower bound in the tournament game differ only by an additive constant of 12.
Thus, the additional constraint that the edges have to be oriented in a particular way
makes it only a little harder for Maker.

Our result seemingly refutes the random graph intuition described above.
However, as a first step towards the proof of Theorem \ref{MainTournament} we will define a suitable, {\em classical} graph game ${\cal G}$ (with no edge-orientations involved), which has the property that every winning strategy for Maker on ${\cal G}$ directly gives her a winning strategy for the tournament game.
It turns out that in the random analogue of ${\cal G}$, the threshold where the game turns from a RandomMaker's win to RandomBreaker's win is around $(2-o(1))\log n$. Thus, from a more subtle point of view, the random graph intuition can be considered valid.

\paragraph{Orientation Games}
We study a variant of the tournament game, which belongs to the class of {\em orientation games}.
Following the notation in \cite{bks2012}, an orientation game is played on the edge set of $K_{n}$ by two players,
called OMaker and OBreaker, which alternately orient (or direct) a previously undirected edge.
At the end of the game, OMaker wins if and only if the final digraph consisting of both OMaker's
and OBreaker's edges satisfies a given property $P$.
Various orientation games have been studied in the literature (see, e.g., \cite{bks2012}, \cite{bs1998} and \cite{chsv1995}).

In this paper we study an orientation-version of the tournament game.
Let $T_k$ be a given tournament on $k$ vertices.
By $Or(T_k)=Or(T_k,n)$ we denote the orientation game in which
OMaker aims to achieve that the final digraph contains a copy of $T_k$.
In the spirit of the $k$-clique game and the $k$-tournament game,
it is quite natural to ask for the largest integer $k_o=k_o(n)$
such that OMaker has a winning strategy for the game $Or(T_k)$ for every tournament $T_k$ on $k_o$ vertices.

Trivially, $k_o$ is at least as large as the corresponding extremal number $k_{t}$ for the ordinary tournament game. We will show that, asymptotically, $k_o$ is at most twice as large as $k_t$.
\begin{theorem}\label{OTournament}
Let $n$ be large enough, let $k\geq 4\log n +2$ be an integer and let $T_k$ be a tournament on $k$ vertices.
Then OBreaker has a strategy to win the game $Or(T_k,n)$.
\end{theorem}
Together with Theorem \ref{MainTournament}, since $k_t\leq k_o$ as mentioned, we therefore get
$$2\log n(1-o(1)) \leq k_o \leq 4\log n (1+o(1)).$$
The rest of the paper is organized as follows. In Section \ref{proofMainTournament}
we prove Theorem \ref{MainTournament}.
In Section \ref{sec:OTournament}
we prove Theorem \ref{OTournament}.
Finally, in Section \ref{remarksAndOpenProblems},
we discuss open problems related to the games we study.

\section{Proof of Theorem \ref{MainTournament}}\label{proofMainTournament}

Let $n \in \nat$ be large enough, and let $k$ be the largest integer such that
$n \geq k 2^{(k+9)/2}$.
Note that by definition, $n < (k+1) 2^{(k+10)/2}$,
so $k \geq 2\, \log n - 2\, \log \log n -12$.
For clarity of presentation, we assume from now on that $n = k 2^{(k+9)/2}$.

Let $T_k$ be the tournament on $k$ vertices that Breaker chooses at the beginning,
with $V(T_k) = \{ u_1,\ldots,u_k\}$.
First, Maker partitions the vertex set into $k$ equally sized parts:
$V(K_n) = V_1 \dot\cup \ldots\dot \cup V_k$. Then she identifies
the class $V_i$ with the vertex $u_i$:
Whenever Maker claims an edge between $V_i$ and $V_j$,
she chooses the direction according to the direction of $u_iu_j$ in $T_k$.
Therefore, her goal reduces to gaining a copy of a clique $K_k$, containing one
vertex from each class $V_i$.
Hence, she plays on the reduced board
\[ X := \Big\{ \{v_i,v_j\} \, : \, v_i \in V_i, v_j \in V_j, i \neq j \Big\}. \]
Our goal is to prove that she wins the classical Maker--Breaker game $(X,\cf)$ where
$\cf$ consists of all edge sets of $k$-cliques in the reduced $k$-partite graph:
\[\cf :=\Bigg\{ \binom{S}{2} : S \subseteq V_1 \,\dot\cup \ldots \dot{\cup}\, V_k \text{ such that  $|S\cap V_i|=1$, for every $1 \leq i \leq k$ }\Bigg\}. \]
To this end, we will use a general criterion for Maker's win from \cite{BeckBook}.
Let us introduce the necessary notation first.
For $p \in \nat$, we define the set of {\em p-clusters} of $\cf$ as
$$\cf_2^p := \left\{ \bigcup_{1\leq i \leq p} E_i \, :
	\{E_1,\ldots,E_p\} \in \binom{\cf}{p}, \, \Big| \bigcap_{1\leq i \leq p} E_i \Big| \geq 2 \right\}.$$

That is, $\cf_2^p$ is the family consisting of all those subsets of $X$ which can be represented
as the union of $p$ distinct winning sets sharing at least two elements of $X$.
Furthermore,
for any family ${\fam=2 H}$ of finite sets,
we consider the well-known potential function used in the Erd\H{o}s-Selfridge criterion
$$T({\fam=2 H}):=\sum_{H\in {\fam=2 H}}2^{-|H|}.$$
According to Beck \cite{BeckBook}, we have the following sufficient condition for Maker's win.

\begin{theorem}[Advanced Weak Win Criterion, \cite{BeckBook}]
\label{awwc}
Maker has a winning strategy for the Maker--Breaker game $(X,\cf)$,
if there exists an integer $p\geq 2$ such that
\begin{equation}\label{AWWC}
\frac{T(\cf)}{|X|} > p + 4p \Big(T(\cf_2^p)\Big)^{1/p}.
\end{equation}
\end{theorem}

In the remainder of the proof we will show that our choice of $(X,\cf)$
satisfies \eqref{AWWC} for $p=4$ and $n$ large enough.
First, we note that
\begin{align}
T(\cf) & = \sum_{F \in \cf} 2^{-|F|} = \left(\frac{n}{k} \right)^k\cdot 2^{-\binom{k}{2}}=2^{5k}\nonumber\\
\text{and} \quad \frac{|X|}{T(\cf)} & = \frac{\binom{k}{2} \left(\frac{n}{k}\right)^2}{2^{5k}}
	\leq \frac{k^2\,  2^{k+9} }{2^{5k}} = o(1). \label{firstPartAwwc}
\end{align}

As a first step towards the application of the Advanced Weak Win Criterion, we give an estimate on $T(\cf_2^4).$
By definition,
\[ \cf_2^4 = \left\{ \bigcup_{1\leq i \leq 4} E_i \, :
	\{E_1,\ldots,E_4\} \in \binom{\cf}{4}, \, \Big| \bigcap_{1\leq i \leq 4} E_i \Big| \geq 2 \right\}. \]
Note that any collection of cliques meets in two edges if and only if it meets in a triangle. Recall that the elements of $\cf_2^4$ are referred to as \emph{clusters}.
Following the standard notation, we call a cluster a {\em sunflower} if there is a triangle such
that any two of the four cliques meet in exactly this triangle. Figure \ref{fig:Sunlfower} shows an illustration.
We denote the subset of sunflowers of
$\cf_2^4$ by $\cS_2^4$.
By definition, a sunflower $F \in \cS_2^4$ has exactly $4\binom{k}{2} - 9$ edges.
In $\cf_2^4$, there are at most $\binom{k}{3}\left(\frac{n}{k}\right)^3\cdot\left(\frac{n}{k}\right)^{4(k-3)}$
sunflowers.
Therefore,
$$T(\cS_2^4) \leq \binom{k}{3}\left(\frac{n}{k}\right)^{4k-9}2^{-4\binom{k}{2}+9} =: f(n,k).$$
It will turn out that $f(n,k)$ dominates the sum $T(\cf_2^4)$.

\begin{figure} [tbp]
 \centering
\phantom{asdfasdfasdfaasdfasfasdfad}
 \includegraphics[width=0.36\textwidth]{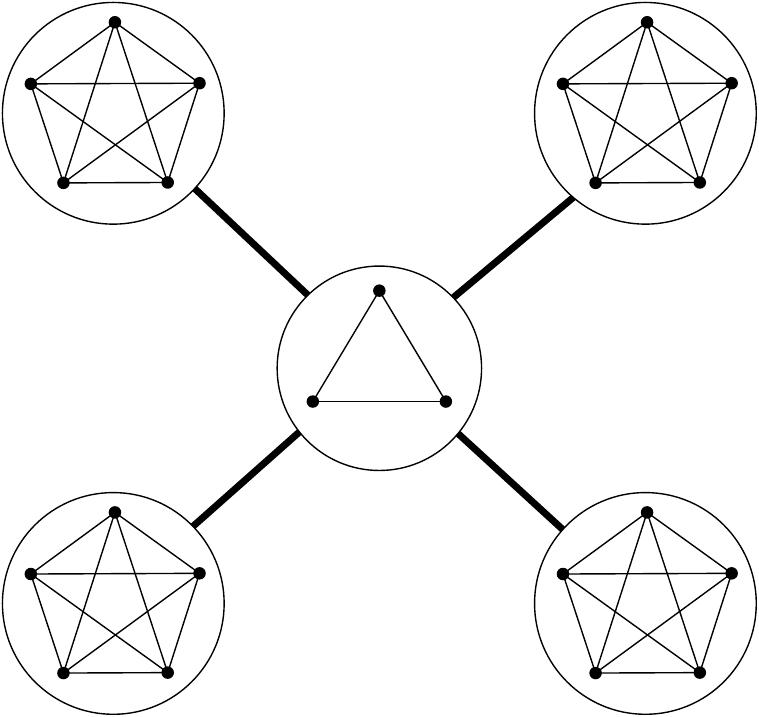}
 \caption{An example of a sunflower for $k = 8$. A thick line indicates that the vertices of the corresponding sets are pairwise connected.}
 \label{fig:Sunlfower}
 \end{figure}

For every $E \in \cf$, we let $V(E)$ denote the set of vertices corresponding to $E$. Note that $|V(E)| = k$ for every $E \in \cf$.
As a first step of our analysis we use the technique of Beck~\cite{BeckBook} to assign to each cluster $F = \bigcup_{1\leq i \leq 4} E_i$ \emph{some} sequence $S(F) := (m_1,m_2,m_3)$ such that
\begin{eqnarray*}
m_1 & = & \left|V(E_1) \cap V(E_2)\right|, \\
	m_2 & = & \left|\left(V(E_1) \cup V(E_2) \right) \cap V(E_3)\right|, \\
m_3 & = & \left|\left(V(E_1) \cup V(E_2) \cup V(E_3)\right) \cap V(E_4)\right|.
\end{eqnarray*}

Note that for a given cluster $F$ we may have several choices to select $S(F)$ (depending on the considered order of the $E_{i}$).
Furthermore, we let $\cf _2^4(m_1,m_2,m_3)$  denote the subset of clusters of $\cf_2^4$
to which we assigned the sequence $(m_1,m_2,m_3)$.
Then obviously,
\begin{equation} \label{eq:clustersum}
T(\cf_2^4)\leq \sum_{m_1=3}^{k}\sum_{m_2=3}^k\sum_{m_3=3}^k T(\cf_2^4(m_1,m_2,m_3)).
\end{equation}

We now bound the cardinality of $\cf _2^4(m_1,m_2,m_3)$.
\begin{proposition} \label{prop:clustercardinality}
For fixed $3 \leq m_1,m_2,m_3 \leq k$,
we have that
\begin{align*}
\left| \cf _2^4(m_1,m_2,m_3)\right| &\leq
	 \binom{k}{3}\, \left(\frac{n}{k}\right)^{4k}\cdot \prod_{j=1}^{3}\binom{jk}{m_j-3}\left(\frac{k}{n}\right)^{m_j}.
\end{align*}
Furthermore, for any cluster $F\in \cf _2^4(m_1,m_2,m_3)$ we have
$
|F| \geq 4\binom{k}{2}-\binom{m_1}{2}-\binom{m_2}{2}-\binom{m_3}{2}.
$
\end{proposition}

\begin{proof} We fix any $m_{1}, m_{2}, m_{3}$ with $3 \leq m_1,m_2,m_3 \leq k$, and we also fix any triple $v_{1}, v_{2}, v_{3}$ of vertices from distinct classes. We now derive an upper bound on the number of those clusters in $\cf _2^4(m_1,m_2,m_3)$ where all four cliques contain $v_{1}$, $v_{2}$, and $v_{3}$. To this end we consider the number of possibilities to select $V(E_{1})\backslash \{v_{1}, v_{2}, v_{3}\}$, $V(E_{2})\backslash \{v_{1}, v_{2}, v_{3}\}$, $V(E_{3})\backslash \{v_{1}, v_{2}, v_{3}\}$, $V(E_{4})\backslash \{v_{1}, v_{2}, v_{3}\}$.
Note that we have $\left(\frac{n}{k} \right)^{k-3}$ possibilities to choose the $k-3$ vertices of $V(E_{1}) \backslash \{v_{1}, v_{2}, v_{3}\}$.

Suppose that for some $1 \leq i \leq 3$ we have already determined the sets $V(E_{1})\backslash \{v_{1}, v_{2}, v_{3}\}$, $\ldots$ ,$V(E_{i})\backslash \{v_{1}, v_{2}, v_{3}\}$. Then $V(E_{1}), \ldots, V(E_{i})$ cover at most $3 + i(k-3)$ vertices.  Therefore, we have at most $\binom{3 + i(k - 3)}{m_{i}-3} \leq \binom{ik}{m_{i}-3}$ choices for those vertices of $\big(V(E_1) \cup \ldots \cup V(E_i)\big) \cap V(E_{i+1})$ which are different from $v_{1}, v_{2}, v_{3}$. Finally, there are at most $\left(\frac{n}{k}\right)^{k-m_{i}}$ possibilities to select $V(E_{i + 1}) \setminus \big(V(E_1) \cup \ldots \cup V(E_i)\big)$.

Therefore, for any given $m_{1}, m_{2}, m_{3}$, every triple  $v_{1}, v_{2}, v_{3}$ of vertices contributes at most
\[ \left(\frac{n}{k}\right)^{k-3}\cdot \prod_{i=1}^{3}\binom{ik}{m_i-3}\left(\frac{n}{k}\right)^{k-m_i} \]
to the number of clusters in $\cf _2^4(m_1,m_2,m_3)$.
Hence,
\begin{align*}
\left| \cf _2^4(m_1,m_2,m_3)\right| &\leq  \binom{k}{3} \left(\frac{n}{k} \right)^{3} \left(\frac{n}{k}\right)^{k-3}\cdot \prod_{i=1}^{3}\binom{ik}{m_i-3}\left(\frac{n}{k}\right)^{k-m_i}   \\
&= \binom{k}{3}\, \left(\frac{n}{k}\right)^{4k}\cdot \prod_{i=1}^{3}\binom{ik}{m_i-3}\left(\frac{k}{n}\right)^{m_i},
\end{align*}
as claimed.
For the second part of the proposition, note that $|E_{1}| = \binom{k}{2}$ and that  every $E_{i+1}$ contributes at least $\binom{k}{2} - \binom{m_{i}}{2}$ new edges to the cluster.
\end{proof}

We now show that $f(n,k)$ dominates the sum $T(\cf_2^4)$.
\begin{lemma}
\label{firstPart}
$T(\cf _2^4)< k^3 \, f(n,k)$, provided $k$ is large enough.
\end{lemma}
\begin{proof}
By definition of $T(\cdot)$ and Proposition \ref{prop:clustercardinality} we have that
\begin{align}\label{TOfSubcluster}
T\big(\cf_2^4(m_1,m_2,m_3)\big)
	&\leq   \binom{k}{3}\, \left(\frac{n}{k}\right)^{4k}
	\cdot \prod_{j=1}^{3}\left(\binom{jk}{m_j-3}\left(\frac{k}{n}\right)^{m_j}\right)
			\cdot 2^{-4\binom{k}{2}+\binom{m_1}{2}+\binom{m_2}{2}+\binom{m_3}{2}}\nonumber\\
	&=  \binom{k}{3} \left(\frac{n}{k}\right)^{4k} 2^{-4\binom{k}{2}}
			\cdot \prod_{j=1}^{3} \binom{jk}{m_j-3}\left(\frac{k}{n}\right)^{m_j}2^{\binom{m_j}{2}}.\nonumber\\	
	&=  f(n,k)\cdot \prod_{j=1}^{3} \binom{jk}{m_j-3}\left(\frac{k}{n}\right)^{m_j-3}2^{\binom{m_j}{2}-3}.
\end{align}
We set $g_j(m) := \binom{jk}{m-3}\left(\frac{k}{n}\right)^{m-3}2^{\binom{m}{2}-3}$.
We will show that
$g_j(m)\leq 1$ for all $j\in\{1,2,3\}$ and $3\leq m\leq k$,
provided $k$ is large enough.
Indeed, for $3 \leq m \leq \frac{15k}{16}$ and $k$ large enough, we have
\begin{align} \label{eq:first}
g_j(m)  & \leq  \left( jk  \right)^{m-3} \cdot 2^{-\frac{k+9}{2}(m-3)} \cdot 2^{\frac{(m+2)(m-3)}{2}} \nonumber \\
	& =   \left(jk \cdot 2^{-\frac{k+9}{2}+\frac{m+2}{2}}\right)^{m-3} 	
	\leq \left(jk \cdot 2^{-\frac{k}{32}-\frac{7}{2}}\right)^{m-3} \leq 1.
\end{align}
For $\frac{15k}{16}\leq m \leq k$ and $k$ large enough, we obtain
\begin{equation} \label{eq:second}
g_j(m) \leq    2^{jk} \left(2^{-\frac{k+9}{2}+\frac{m+2}{2}}\right)^{m-3}
   	\leq 2^{3k} \left(2^{-\frac{7}{2} }\right)^{m-3}
  	\leq 2^{3k} \left(2^{-\frac{7}{2} }\right)^{\frac{15k}{16}-3}  \leq 1.
\end{equation}

Now, \eqref{TOfSubcluster}, \eqref{eq:first} and \eqref{eq:second} imply that $T(\cf_2^4(m_1,m_2,m_3))\leq f(n,k)$
for any sequence $(m_1,m_2,m_3)$, provided $k$ is large enough. Due to \eqref{eq:clustersum}, we conclude that
$T(\cf_2^4)\leq k^3 f(n,k).$
\end{proof}

Finally, we show that the Advanced Weak Win Criterion (Theorem \ref{awwc}) applies with $p=4$.
\begin{corollary}
For $n$ large enough, $T(\cf)>16|X|\Big(\left(T(\cf _2^4)\right)^{1/4}+\frac{1}{4}\Big)$.
\end{corollary}
\begin{proof}
By Lemma \ref{firstPart}, the definition of $f(n,k)$ and the fact that $|X| \leq n^{2}$ we get that
\begin{align*}
\frac{16\, |X|\, (T(\cf_2^4))^{1/4}}{T(\cf)}
	 & \leq \frac{16\, |X|\, (k^3 \, f(n,k))^{1/4}}{T(\cf)}
      \leq \frac{16\, n^{2}\, k^{\frac{3}{4}}\, \left(k^3 \,  \left(\frac{n}{k}\right)^{4k-9}\, 2^{-4 \binom{k}{2} + 9} \right)^{\frac{1}{4}}}{T(\cf)} \\
    &  \leq \frac{16 \cdot 2^{\frac{9}{4}}\, n^{2}\, k^{\frac{6}{4}}\,  \left(\frac{n}{k}\right)^{k-\frac{9}{4}}\, 2^{- \binom{k}{2}}}{\left(\frac{n}{k} \right)^{k}\cdot 2^{-\binom{k}{2}}}
    \leq \frac{ 100 n^2\, k^{\frac{6}{4}} }{\left(\frac{n}{k} \right)^{\frac{9}{4}}}  \\
	& = \frac{100 k^{\frac{15}{4}}}{n^{\frac{1}{4}}} < 100\cdot 2^{-\frac{k}{8}+\frac{14}{4}\log(k)}= o(1).
\end{align*}
By \eqref{firstPartAwwc},
\[ \frac{\frac{1}{4} \cdot 16 |X|}{T(\cf)} = o(1), \]
which concludes the proof.
\end{proof}

\section{Proof of Theorem \ref{OTournament}}\label{sec:OTournament}
We first need some notation.
In an $(a:b)$ Maker--Breaker game $(X,\cf)$
Maker claims $a$ elements and Breaker claims $b$ elements in each round. A game is called \emph{biased} if $a \neq 1$ or $b \neq 1$.

In order to provide OBreaker
with a winning strategy for the game $Or(T_k)$
we associate with $Or(T_k)$ an auxiliary \emph{biased} Maker--Breaker game.
In the first step of the proof we show that Breaker has a strategy to win the auxiliary game, and in the second step we prove that this strategy directly gives him a winning strategy for $Or(T_k)$.

We will make use of the generalized Erd\H{o}s-Selfridge-Criterion proven by Beck \cite{BeckBook}.

\begin{theorem}[Generalized Erd\H{o}s-Selfridge-Criterion]\label{ESC}
Let $X$ be a finite set and let ${\mathcal F} \subseteq 2^X$. If
$$\sum_{F \in {\mathcal F}}(1+b)^{-|F| / a} < \frac{1}{1+b},$$ then
Breaker has a winning strategy in the $(a : b)$ Maker--Breaker game $(X, {\mathcal F})$.
\end {theorem}

Let $T_k$ be some tournament on $k$ vertices.
Consider the $(2:1)$ Maker--Breaker game $\cH(T_k)= {\fam=2 H}(T_k,n) = (X,\cf (T_k))$ where
$$X:=\left\{(u,v):\, u,v\in\, V(K_n),\, u\neq v \right\}$$
is the board of the game consisting of $|X|=n(n-1)$ elements, and
$${\mathcal F (T_k)}:=\left\{ S\subseteq X:\, S \text{ is a copy of } T_k \right\}$$
is the family of winning sets.

\begin{claim}\label{BreakerWin}
For large enough $n$ and $k\geq 4\log n + 2$,
Breaker has a winning strategy in
${\fam=2 H}(T_k)$, for any tournament $T_k$ on $k$ vertices.
\end{claim}

\begin{proof}
We check that Theorem \ref{ESC} applies.
By definition,
$|{\fam=2 F}(T_k)|
<n^k$,
and $|F|=\binom{k}{2}$ for every $F\in{\fam=2 F} (T_k)$. Therefore,
and since $k \geq 4 \log n + 2$,
$$
\sum_{F \in {\mathcal F}}2^{-|F| / 2}
< n^k\cdot 2^{-k(k-1)/4} \leq \frac{1}{2}.
$$
\end{proof}

We conclude the proof of Theorem \ref{OTournament} with the following lemma.
\begin{lemma}\label{hypergraph-game}
Let $T_{k}$ be a tournament on $k$ vertices.
Suppose that Breaker has a strategy to win the game ${\fam=2 H}(T_k)$.
Then there is also a winning strategy for OBreaker in the game  $Or(T_k)$.
\end{lemma}

\begin{proof}
We first need some notation. By \emph{directing an edge $(u,v)$} we mean that we direct the edge spanned by $u$ and $v$ from $u$ to $v$. We note that in each round of ${\fam=2 H}(T_k)$, Maker is allowed to choose either two, one, or zero elements. (Otherwise she can just claim additional, arbitrary elements, and then follow her strategy. If this strategy calls for something she occupied before, she takes an arbitrary element; no extra element is disadvantageous for her.)

Suppose, for a contradiction, that OMaker has a winning strategy $\cS$ for $Or(T_k)$.
We now describe a strategy $\cS'$ for Maker in ${\fam=2 H}(T_k)$.
During the play, Maker simulates (in parallel) a play of the game $Or(T_k)$, and maintains the invariant that after each of her moves in ${\fam=2 H}(T_k)$, every pair $u,v \in V$ has the property that
\begin{itemize}
\item[$(i)$] in $\cH(T_k)$, if Breaker owns the element $(u,v)$ then Maker owns $(v,u)$, and
\item[$(ii)$] in $Or(T_k)$, there is a directed edge from $u$ to $v$ if and only if Maker has claimed the element $(u,v)$ in $\cH(T_k)$.
\end{itemize}

Let $(a,b)$ be the edge $\cS$ tells OMaker to direct in her first move.
Then Maker claims the element $(a,b)$ in the actual game $\cH(T_k)$
(at this point she does not make use of the possibility to occupy two elements),
and directs $(a,b)$ in the parallel game $Or(T_k)$ as OMaker.

Suppose that $i$ rounds have been played, and let $(u,v) \in X$ denote the element Breaker chose in his $i$th move.
If Maker has already claimed $(v,u)$ in a previous round then she does not claim a single element.
Otherwise, as her $(i+1)$st move, she first occupies the element $(v,u)$. In the parallel game $Or(T_k)$,
she directs as OBreaker the edge $(v,u)$. Then she identifies the edge $(x,y)$ the strategy $\cS$ tells OMaker to direct.
Finally, she directs $(x,y)$ as OMaker in $Or(T_k)$, and claims the element $(x,y)$ in $\cH(T_k)$.

We note that the invariants (i) and (ii) remain satisfied after Maker's $(i+1)$st move. So, by following $\cS'$, Maker can guarantee that at the end of the game, these invariants still hold. Since by assumption $\cS$ is a winning strategy, the final digraph in $Or(T_k)$ contains a copy of $T_{k}$. Together with invariant (ii) this yields that Maker possesses all elements of some winning set in ${\fam=2 H}(T_k)$.
This contradicts the assumption that Breaker has a winning strategy for ${\fam=2 H}(T_k)$.
\end{proof}

\section{Concluding remarks and open problems}\label{remarksAndOpenProblems}

\paragraph{Random Graph Intuition.}
As noted in the introduction, our lower bound on $k_t$ seemingly refutes the random graph intuition.
However, to prove Theorem \ref{MainTournament},
we reduced the tournament game to a classical Maker--Breaker game on the board of the complete $k$-partite
graph with vertex classes $V_1,\ldots,V_k$, where each $V_i$ has size roughly $n/k$.
Let us consider the corresponding random game on this reduced board, where in every round, each player claims a random unclaimed edge.
We shortly sketch why the threshold where this game turns from a RandomMaker's win to a RandomBreaker's win is around $(2 - o(1)) \log n$:
By standard techniques it can be shown that the expected number of $k$-cliques in RandomMaker's graph is
\[(1+o(1))\left( \frac{n}{k}\right)^k 2^{-\binom{k}{2}},\] which jumps from
below one to above one at $k = 2\log n - 2\log \log n - 1 +o(1)$.
Analogously to the proof of the concentration result for the largest clique size
in the random graph $G(n,m)$ \cite{BollobasBook}, it can be shown that
if
$\left\{\begin{array}{l}
\text{$k\leq (2-o(1))\log n$ } \\
\text{$k\geq (2-o(1))\log n$ }
\end{array}\right\}$
then RandomMaker's graph
$\left\{\begin{array}{l}
\text{contains } \\
\text{does not contain}
\end{array}\right\}$
a $k$-clique a.a.s.
\vspace*{0.2cm} Thus, from a more subtle point of view, the random graph intuition can be considered valid.

\paragraph{Tournament game.}
The strategy for Maker in Theorem \ref{MainTournament} is
independent of the actual tournament $T_k$ Breaker chooses.
On the other hand, the upper bound in \eqref{eq:upperboundktour}
is the upper bound from the $k$-clique game. That means
that for $k> 2\log n - 2\log \log n +o(1)$, Breaker has a strategy to
prevent Maker from building {\em any} given tournament on $k$
vertices. So both, the lower and the upper bound, are universal for
all tournaments on $k$ vertices.
On the other hand, there is a tournament on
$k_{cl}$ vertices which Maker can build:
Consider the transitive tournament on the vertex set $\{u_1,\ldots, u_k\}$,
where, say, for all indices $1 \leq i < j \leq k$ the edge between $u_{i}$
and $u_{j}$ is directed from $u_{i}$ to $u_{j}$.
Now, let $\{v_1,\ldots,v_n\}$ be an arbitrary enumeration of the vertices in $K_n$.
Then Maker can just follow the strategy provided by the $k$-clique game.
Whenever this strategy tells her
to claim the edge $\{v_i,v_j\}$ for $i<j$, she chooses the direction $(v_i,v_j)$.
It would be interesting to determine whether all tournaments are ``equally hard''
for Maker.
Therefore, we pose the following question.
\begin{problem}
Does there exist a tournament $T_{k}$
on $k\leq k_{cl}$ vertices such that in the $T_k$-building game,
Breaker has a strategy to prevent Maker from building $T_k$?
\end{problem}

Note that a negative answer to this question, together with Theorem \ref{BeckCliqueTheo},
would give us the exact value of $k_t$. But even if there exists a tournament
on at most $k_{cl}$ vertices which Breaker can prevent,
it is of particular interest to get rid of the gap in the constant term.

\begin{problem}
Determine the exact value of $k_{t}$.
\end{problem}


\paragraph{Orientation tournament game.}
A similar discussion arises for the orientation version $Or(T_k)$.
Since we use Breaker's strategy in the game $\cH (T_k)$ as a black box,
by Claim \ref{BreakerWin}, OBreaker wins $Or(T_k)$ for $k\geq 4\log n +2$, for
{\em any} tournament $T_k$ on $k$ vertices.
\begin{problem}
For $k_1 \in \nat$, do there exist two (non-isomorphic) tournaments $T$ and $T'$
on $k_1$ vertices such that OMaker has a winning strategy in the orientation game $Or(T)$,
but OBreaker has a strategy in the game $Or(T')$?
\end{problem}
Furthermore, the lower and the upper bound
for $k_o(n)$ are a factor of two apart.
The probabilistic analysis would suggest the breakpoint to be around $2 \log n$.
\begin{problem}
Determine the constant $2\leq c \leq 4$ such that $k_o(n) = (c+o(1)) \log n$.
\end{problem}

{\bf Universal tournament game.}
As noted in the introduction, Maker has a strategy to occupy a copy
of every tournament on $k$ vertices for $k\leq (1/2 -o(1)) \log n$.
Beck conjectured that this result is not best possible and posed the following problem.
\begin{problem}[\cite{BeckBook}, p. 457]
Determine the largest $k_u$ such that Maker has a strategy such that
at the end of the game, her digraph contains {\em every} tournament on $k_u$
vertices.
\end{problem}
For the upper bound, we can show that
$k_u \leq (1+o(1))\log n$:
We note that the number of unlabelled $k$-tournaments is at least $c(k) := 2^{\binom{k}{2}}/k! > 2^{\binom{k}{2}-k\log k}$.
By assumption, Maker has a strategy to occupy a copy of every tournament on $k_{u}$ vertices.
Hence, at the end of the game, the underlying graph of Maker's graph contains $c(k_{u})$ (not necessarily edge-disjoint) distinct cliques.
However, a result of Bednarska and \L uczak (see Lemma 5 in \cite{bl00}) asserts that there is some $k = (1 + o(1))\log n$ such that in the ordinary graph game (where no edge-orientations are involved), Breaker has a strategy to prevent Maker from claiming more than $c(k)$ distinct $k$-cliques. Thus, $k_u \leq (1+o(1))\log n$. To the best of our knowledge, nothing better is known for the universal
tournament game.

\bibliographystyle{abbrv}
\bibliography{AnitaReferences,TournamentRefs}

\end{document}